\documentclass[12pt]{article}
\usepackage{amsmath,amsthm,amssymb}

\newtheorem{theorem}{Theorem}[section]
\newtheorem{lemma}[theorem]{Lemma}

\numberwithin{equation}{section}

\newcommand{\ssection}[1]{%
  \section[#1]{\normalsize\bf #1}}

\begin{document}

\title {\large\sc Uniform error estimates for Navier-Stokes flow\\
with an exact moving boundary\\ using the immersed interface method
 }

\author{{\normalsize J. Thomas Beale} \\
{\normalsize{\em Department of Mathematics, Duke University, Box 90320,
Durham,}}\\
{\normalsize{\em North Carolina 27708, U.S.A.}}\\
{\normalsize{\tt beale@math.duke.edu}}}

\date{}

\maketitle

{\footnotetext {
This work was supported in part by National Science Foundation
grant DMS-1312654.  } }


\begin{abstract}
We prove that uniform accuracy of almost second order can be achieved
with a finite difference method applied to Navier-Stokes flow at low
Reynolds number with a moving boundary, or interface, creating jumps in the velocity
gradient and pressure.  Difference operators are corrected to $O(h)$
near the interface using the immersed interface method, adding terms
related to the jumps,
on a regular grid with spacing $h$ and periodic boundary conditions.
The force at the interface is assumed known within an error tolerance;
errors in the interface location are not taken into account.  The 
error in velocity is shown to be uniformly $O(h^2|\log{h}|^2)$, even
at grid points near the interface, and, up to a constant, the pressure has error
$O(h^2|\log{h}|^3)$.  The proof uses estimates for finite difference
versions of Poisson and diffusion equations which exhibit a gain in 
regularity in maximum norm.
\end{abstract}   

\newcommand{\beq}{\begin{equation}}
\newcommand{\eeq}{\end{equation}}
\newcommand\p{\,+\,}
\newcommand\lee{\,\leq\,}
\newcommand\gee{\,\geq\,}
\newcommand\m{\,-\,}
\newcommand\eq{\,=\,}
\newcommand{\eps}{\varepsilon}
\newcommand{\sig}{\sigma}
\newcommand{\pa}{\partial}
\newcommand{\lilhalf}{{\textstyle \frac12}}
\newcommand{\lilth}{{\textstyle \frac32}}
\newcommand{\mm} {_{\max}}
\newcommand{\np}{{n+1}}
\newcommand{\nm}{{n-1}}
\newcommand{\nhalf}{{n+1/2}}
\newcommand{\ex}{{ex}}
\newcommand{\NN}{{\mathcal N}}
\newcommand{\R}{{\mathbb R}}
\newcommand{\Z}{{\mathbb Z}}
\newcommand{\laph}{\Delta_h}
\newcommand{\Ptw}{{\tilde P}}
\newcommand{\fcns}{{\mathcal F}(\Omega_h)}
\newcommand{\phihat}{{\hat \varphi}}
\newcommand{\fhat}{{\hat f}}
\newcommand{\ghat}{{\hat g}}

\ssection{Introduction}
Recently there has been enormous development in numerical methods for
fluid flow with moving boundaries or fluid-structure interaction.  Often finite
difference methods are used on a Cartesian grid which does not conform to the
moving boundary.  A separate representation is used for the boundary, and the effect
of the boundary on the fluid must be included.  For biological models 
practical applications have most often used the immersed boundary method \cite{pacta}
in which the force on the fluid from the boundary is spread to nearby grid points.
Other methods maintain a sharp interface and are seen to attain about $O(h^2)$
accuracy in the velocity.  Generally these methods are carefully designed to
control the truncation error, taking into account the location of the
boundary relative to the grid cells.  Here we focus on the immersed interface
method (\cite{khoo,leelev,ll97,libook,lilai,mayo84,xuwang2d,xuwangsys,wiegbube})
in which difference operators for the velocity and pressure
are corrected where the stencil crosses the interface using jumps in the quantities
and their derivatives.  Closely related methods use ghost points or cut cells.
With low to moderate Reynolds number, it is often observed in computations
that the velocity error is about $O(h^2)$ even when the truncation error is
$O(h)$ near the interface, but this phenomenon has not been explained, and
understanding of the solution error has come mainly from numerical evidence.

In this work we estimate the errors in velocity and pressure, uniformly
with respect to grid points, including those near the moving boundary, in a simple
prototype problem using the immersed interface method.  We neglect possible errors
in the boundary location and concentrate on the errors in fluid variables brought
about by the numerical treatment of the force from the moving boundary.
Thus for a problem with coupled motion
of the fluid and moving boundary we are only partially accounting for the errors. 
We verify analytically the
principle that $O(h)$ truncation error at the moving boundary can lead to uniform
accuracy close to $O(h^2)$.  In doing so we elucidate the minimal features necessary
to achieve this accuracy.  This result depends on the effect of diffusion with
implicit time stepping and thus is significant at low Reynolds number.

We will always measure errors in maximum, or $L^\infty$, norm.
One reason is that methods in use are
generally designed to control maximum truncation errors near the moving boundary.  A second reason
is that the errors in the solution are likely to be largest near the boundary, and the
most meaningful measure of the error is a uniform estimate.
We use estimates derived in \cite{bealay,smooth} for discrete Poisson and diffusion
equations with a gain in regularity in maximum norm.
Although we have chosen to study the immersed interface method, we hope that 
the analytic technique introduced here will be suggestive for the larger class of related methods.

We first state the physical problem.
We consider fluid flow in a rectangular region $[-L,L]^d$ in dimension $d = 2$ or $3$,
with velocity $u$ and pressure $p$, both periodic.  We suppose the moving boundary or interface
$\Gamma$ is a closed curve in $\R^2$ or closed surface in $\R^3$.
We assume the density is constant and the Reynolds number is low to moderate, and
for simplicity we set both to $1$.  The fluid flow is determined by the Navier-Stokes
equations for the velocity and pressure, with a force exerted by the interface on the fluid,
\beq  u_t + u\cdot\nabla u + \nabla p = \Delta u + f \delta_\Gamma \,, \quad
          \nabla\cdot u \eq 0  \eeq
where $f$ is the force density, $\delta_\Gamma$ is the delta function
restricting to the surface $\Gamma$, and $\Delta = \nabla^2$ is the Laplacian.
Equivalently, the equation holds away from $\Gamma$
with zero force, and the velocity and pressure have the jumps across $\Gamma$
\beq  [u] \eq 0 \,,  \qquad \left[ \frac{\pa u}{\pa n} \right] 
             \eq - f_{{tan}} \,,\quad
                   f_{{tan}} \,\equiv\,  f - (f\cdot n)n   \eeq
\beq  [p] \eq f\cdot n \,, \qquad \left[ \frac{\pa p}{\pa n} \right] =
          \nabla_\Gamma \cdot f_{{tan}}   \eeq
Here $n$ is the outward unit normal at $\Gamma$, and $[p] = p_+ - p_-$ is the
difference between the outside and inside values at $\Gamma$.  (See e.g.
\cite{khoo,ll97,libook,lilai,pp,xuwangsys}.)

  The operator
$\nabla_\Gamma\cdot\;$ is the surface divergence;
in $\R^2$ it is just the arclength derivative.
(E.g, see \cite{aris}, (9.41.1) for the definition of $\nabla_{\Gamma}\cdot\;$
and \cite{xuwangsys} for a thorough derivation of $[{\pa p/\pa n}]\;$.)
The fact that the pressure $p$ is periodic depends on the
facts that
\beq   [u\cdot\nabla u]\cdot n \eq 0 \quad \mbox{and} 
          \quad \int_\Gamma \left[ \frac{\pa p}{\pa n}\right] \,dS \eq 0 \eeq
At each $t$, the pressure has an indefinite constant; adding $p_0(t)$ to
$p(x,t)$ does not change (1.1) or (1.3).

Typically $\Gamma$ moves with the fluid velocity and the force $f$ is determined from the configuration of $\Gamma$, depending on its
material properties, e.g. elastic forces, so that $f$ and $\Gamma$ depend on the
fluid variables.  In this work we assume the location of $\Gamma$ is known exactly
and $f$ is known within a certain error tolerance; see Theorem 1.1 below.
Thus, for the full problem, we estimate only the part of the error in velocity and
pressure from their direct computation while neglecting the influence of errors in $\Gamma$.
With this qualification, the maximum errors in velocity and pressure in the scheme
studied here are shown to be $O(h^2 |\log{h}|^2)$ and $O(h^2 |\log{h}|^3)$, resp.

We discretize $u$, $p$ and their derivatives on a regular grid at points
$x_j = (j_1,j_2)h$ or $x_j = (j_1,j_2,j_3)h$ with $h = L/N$.  We use the usual centered
differences for the discrete gradient $\nabla_h$, divergence $\nabla_h\cdot\;$ and
Laplacian $\Delta_h$.  All are $O(h^2)$ accurate for smooth functions, and thus at grid
points where the stencil does not cross the interface $\Gamma$.  In the immersed interface
method, the differences are corrected at the irregular points using jumps at $\Gamma$
for the variables and their partial derivatives.
For example, for a function $v(x)$, $x \in \R$, if $x_{j-1}$, $x_j$ are inside $\Gamma$ and
$x_{j+1}$ is outside, $\Gamma$ intersects the grid line
at $x^*$ with $x_j \leq x^* \leq x_{j+1}$, and $h_+ = x_{j+1} - x^*$,    then
\beq  v_x(x_j) \eq \frac{v_{j+1} - v_{j-1}}{2h} - \frac{1}{2h}
       \left( [v]  + h_+[Dv] + \frac{h_+^2}{2}[D^2v] \right) + O(h^2) \eeq
\beq  v_{xx}(x_j) \eq \frac{v_{j-1} - 2v_{j} + v_{j+1}}{h^2} - \frac{1}{h^2}
       \left( [v]  + h_+[Dv] + \frac{h_+^2}{2}[D^2v] \right) + O(h) \eeq
provided $v$ is $C^3$ on each side,
with similar but different formulas at $x_{j+1}$; e.g., see \cite{libook,wiegbube,xuwangsys}.
Jumps in the first and second partial derivatives of $u$ and $p$,
needed for corrections here, can be found from
(1.2), (1.3), as explained in the works cited.

Next we present the scheme to be analyzed.
We choose a time step $\tau = O(h)$ and compute the velocity $u^n$ at time $t_n = n\tau$.
In updating $u^n$ to $u^\np$, we will assume the force $f$ and needed corrections
are known up to time $t_\np$.  (If the interface $\Gamma$ is updated explicitly, the
simplest possibility, then $\Gamma^\np$ is found from $u^n$, and it determines the jumps at
time $t_\np$.)
Assuming $u^0$ is given, we start with
\beq  u^{1} - u^0 \eq - \tau  (u\cdot\nabla u)^{0} - \tau \nabla p^{0} +
  \tau \Delta u^{1/2} + \tau C_1  + \tau C_7 \eeq
Then, with $u$ known up to time $t_n$, $n\geq 1$,
the new velocity $u^\np$ is found from
\beq  u^{n+1} - u^n \eq - \tau  (u\cdot\nabla u)^{n+1/2} - \tau \nabla p^{n+1/2} +
  \tau \Delta u^\nhalf + \tau C_1  + \tau C_7 \eeq
We will describe the discretization of each term.
Here $C_1$ corrects the approximation $u^\nhalf_t \approx (u^\np - u^n)/\tau$
at grid points which are crossed by the interface during the interval
$t_n \leq t \leq t_\np$.  This correction is a term proportional to $[u_t]$.   Since
the velocity is continuous at the interface, the material derivative
$u_t + u\cdot\nabla u$ is also continuous, so that $[u_t] = - u\cdot[\nabla u]$.
(E.g. see \cite{khoo}, (33)-(35) or \cite{xuwangsys}, Cor. 3.2.)
Corrections to other terms at locations crossed by the interface during a time interval
will be included in $C_7$, discussed later.

The advection term is extrapolated in time for $n \geq 1$,
\beq  (u\cdot\nabla u)^{n+1/2} \eq
       \lilth u^n\cdot\nabla_h u^n - \lilhalf u^\nm\cdot\nabla_h u^\nm + C_2  \eeq
where $C_2$ is the correction to the centered difference
$\nabla_h$, determined from the jumps
$[Du]$ and $[D^2 u]$, at each time
$t_n$ and $t_\nm$ as in (1.5).  Similarly
\beq   \Delta u^\nhalf = \lilhalf \left( \Delta_h u^\np + \Delta_h u^n \right) + C_3 \eeq
with corrections $C_3$ to $\Delta_h$ again determined from $[Du], [D^2 u]$ at each time
as in (1.6).

For the pressure term we first compute the divergence of (1.9),
\beq  \nabla\cdot(u\cdot\nabla u)^\nhalf = \nabla_h\cdot (u\cdot\nabla u)^{n+1/2} + C_4 \eeq
where $C_4$ corrects $\nabla_h\cdot\;$ for the jumps in $[Du], [D^2 u]$.
We then solve the discrete Poisson problem
\beq  \Delta_h p^\nhalf \eq - \nabla_h\cdot(u\cdot\nabla u)^{n+1/2} - C_4 + C_5  - m \eeq
with periodic boundary conditions.  Here
$C_5$ corrects $\Delta_h p^\nhalf$ using the jumps $[p], [Dp], [D^2p]$ and $[u\cdot\nabla u]$.
The last term $m$ is the mean value, or average, of $-C_4 + C_5$ on the grid.
It is subtracted so that the right side
above has mean value zero and thus is in the range of $\Delta_h$; we will see that the
error resulting from $m$ is not significant.  The periodic solution of (1.12) has an
indefinite constant; we choose it to have mean value zero. Finally,
\beq  \nabla p^{n+1/2} = \nabla_h p^\nhalf + C_6  \eeq
where $C_6$ corrects $\nabla_h$ using $[p], [Dp], [D^2p]$.

For each of the terms $(u\cdot\nabla u)^{n+1/2}$, $\nabla p^{n+1/2}$, $\Delta u^\nhalf$ we
compute separately at two time levels, including corrections.  At a grid point crossed by the interface during the time interval we add a correction proportional to the jump in each quantity; see (14),(15) in \cite{khoo}.  These terms form $C_7$. 

We will show that this scheme produces values of $u$ and $p$ that have accuracy slightly
less than $O(h^2)$ with certain assumptions.  We summarize the conclusion as follows.

\begin{theorem} Suppose the exact solution of (1.1)-(1.3) is smooth in space-time
on each side of the interface $\Gamma$ for $0 \leq t \leq T$, and also $\Gamma$ is smooth.
We neglect any errors in $\Gamma$, and we assume that
$f$ and $\nabla_{tan} f$ are known within maximum error $O(h^2)$.
Then, with $\tau/h$ constant and $h$ sufficiently small,
\beq  \max_{j,n} { \left| u^{computed}(x_j,t_n) - u^{exact}(x_j,t_n)\right| }
                \,\leq\, K_T h^2 |\log{h}|^2 \eeq
for $t_n = n\tau \leq T$ and some constant $K_T$ independent of $h$.  The pressure $p^n$
can be found as above at time $t_n$ so that, for some constant $p_0^n$, depending on $h$,
$p^n + p_0^n$ differs from the exact pressure with maximum error
bounded by $K_T h^2 |\log{h}|^3$.
\end{theorem}

We have assumed periodic boundary conditions for the computational domain
to avoid the serious issue of handling the boundary conditions, in order to focus
on the accuracy near the interface.  The difficult question of
combining solid boundary conditions with the pressure solution has been dealt with
extensively for finite difference methods including projection methods
(e.g. \cite{bcm,pm3,shenover,llp}).
We expect that in principle the two issues are separate,
provided the interface is away from the outer boundary.  The periodic condition is
helpful for the analysis in that various operators commute in this case.

Often a MAC, or staggered, grid is used for velocity and pressure, in the present problem
(\cite{khoo,xuwang2d}) and others.  It allows more natural treatment of computational boundaries.
A second advantage is that the discrete version of the projection on divergence-free vector fields is
an exact projection.  On the other hand, the simplicity of a single grid, as in the present case,
is a desirable advantage.  We expect that a result similar to the present one would hold
using a periodic MAC grid.  A pressure increment scheme, updating the pressure rather than
finding it from (1.12), might be used, as in \cite{khoo,lilai}.  Such a scheme would be different from
this one because of the effect of the discrete projection \cite{abc,guyap} and the present analysis would not apply directly.

For temporal discretization of the diffusion
we have chosen to use the Crank-Nicolson (CN) method since it is the
most familiar second-order accurate method allowing time step $O(h)$ with viscosity and since
it is often used with interface simulations and with the projection method.
CN has an important smoothing property in maximum
norm, proved in \cite{khoo}, related to $A$-stability.  This smoothing is largely responsible for the
result stated above.  However, other methods (called $L$-stable) such as BDF2 have better smoothing
properties.  The result proved here should also be true for these methods and perhaps can be improved.

A disadvantage with the use of a single grid
is that $(\nabla_h\cdot)\nabla_h \neq \Delta_h$.
Instead $(\nabla_h\cdot)\nabla_h  = \Delta_0$, where $\Delta_0$
is the ``wide Laplacian'', a sum of second differences such as
\beq  v_{xx}(x_j) \,\approx\, (4h)^{-2} \left( v_{j-2} - 2v_{j} + v_{j+2}\right) \eeq
Consequently the discrete version $\Ptw = I - \nabla_h(\Delta_h)^{-1}\nabla\cdot\;$
of the projection onto divergence-free
vector fields is only an approximate projection, i.e., $\Ptw^2 \neq \Ptw$ 
(cf. \cite{abc,guyap}).  We will see that $\Ptw$ enters in expressing
the error in $u$.  We will find that the exact discrete projection $P_0$,
defined with $(\Delta_0)^{-1}$ rather than $(\Delta_h)^{-1}$, is useful in our estimates. 

In Sec. 2 we collect facts about difference operators in maximum norm relevant to this work.
We discuss $\Ptw$ and $P_0$.  We give an estimate for a discrete Poisson problem with gain of
regularity.  We state estimates for CN time steps, also with gain in regularity.  We state a
lemma which allows a grid function near the interface to be estimated in a lower norm with a gain
of a factor of $h$.  In Sec. 3 we classify the truncation errors made by the exact solution in satisfying the scheme.  We can allow errors which,
for example, are first differences of quantities $O(h^2)$.
In Sec. 4 we estimate the growth of errors in the velocity using the results of Secs. 2 and 3 and
verify the conclusion above.  Finally in the Appendix we give a criterion for boundedness of a discrete linear operator in maximum norm and show that a certain operator relating $\Ptw$ and $P_0$ is bounded.

Previous analysis of Cartesian grid methods with interfaces has concerned elliptic problems
(\cite{bealay,libook,ghostconv}) and linear diffusion (\cite{smooth}, Sec. 8).
The effect of discrete projections on accuracy with irregular boundaries
was studied in \cite{guyforce}.
The case of boundary conditions
at irregular boundaries in Cartesian grids, rather than interfaces, is quite different
and has a long history;
e.g. see \cite{mm}, Ch. 6.  Analysis of finite difference methods for the Navier-Stokes equations
usually assumes smooth, rather than piecewise smooth, solutions and uses $L^2$ estimates.

We use the letter $K$ for generic constants independent of $h$.  $D$ will be any first 
spatial derivative, and $D_h$ any difference operator, not necessarily centered.  However,
$\nabla_h$ in a gradient or divergence will always mean the centered difference.

\ssection{Operators on periodic grids}
We collect facts about
several difference operators on periodic grid functions that will be used in the
following arguments.  Let $\Omega_h$ be
the set of grid points $x_j = jh$ in $\R^d$ where $j$ is a $d$-tuple of integers with 
$|j_\nu|h \leq L$, $1 \leq \nu \leq d$, and let $\fcns$ be the space of periodic
grid functions on $\Omega_h$.  We always use the maximum norm for such functions:  
For $w \in \fcns$, $\|w\| = \max_{x_j \in \Omega_h} |w(x_j)| $.  Correspondingly,
for a linear operator ${\mathcal L}$ on $\fcns$ or a subspace,
$\|{\mathcal L}\| = \max \|{\mathcal L}w\|$ for
$\|w\| = 1$.
Of course we are primarily interested in how these norms depend on $h$.
It will be helpful that the difference operators and inverses we deal with
commute because of the periodicity.

The Fourier modes
\beq  e_k(x_j) \eq e^{ikjh\pi/L}\,, \qquad k \in \Z^d\,, \quad -L/h < k_\nu \leq L/h \eeq 
form a basis of $\fcns$.  They are eigenfunctions
for $\Delta_h$ and the ``wide Laplacian'' $\Delta_0$,
\beq  \Delta_h e_k = \sigma(kh) e_k\,, \qquad \sigma(kh) = 
       - \frac{4}{h^2}\sum_{\nu=1}^d \sin^2 \frac{kh\pi}{2L}  \eeq 
\beq  \Delta_0 e_k = \sigma_0(kh) e_k\,, \qquad \sigma_0(kh) = 
       - \frac{1}{h^2}\sum_{\nu=1}^d \sin^2 \frac{kh\pi}{L}  \eeq 
Evidently the null space of $\Delta_h$ consists of constant functions, the multiples
of $e_0$; the null space ${\mathcal N}_0$ of $\Delta_0$ has dimension $2^d$ and is spanned by
$e_k$ with each $k_\nu = 0$ or $L/h$.  (Cf. \cite{abc}.)
Each operator is invertible on the subspace
of $\fcns$ spanned
by the remaining modes, which is also the range of the operator.  We will call these
subspaces $X_h$ and $X_0$, respectively.  Note that
$ X_h \eq \{w \in \fcns: \Sigma_j\; w(x_j) = 0 \}   $,
the subspace with mean value zero.
$\Delta_h$ is invertible on $X_h$, and we will
write $(\Delta_h)^{-1}$ on $X_h$.  Similarly we have $(\Delta_0)^{-1}$ on $X_0$.
If $D_h$ is any centered first difference and $w \in \fcns$ then $D_h w \in X_0$ since
$D_h$ is zero on ${\mathcal N}_0$.  Thus $(\Delta_0)^{-1}D_h$ and  $(\Delta_h)^{-1}D_h$ 
are meaningful for any centered $D_h$.

Next we discuss the two discrete versions of the projection on divergence-free vector fields.
The ``exact discrete projection'', again using centered differences $\nabla_h$, is
\beq  P_0 v \eq v - \nabla_h  (\Delta_0)^{-1} \nabla_h\cdot v \eeq 
Since $\nabla_h\cdot\nabla_h = \Delta_0$, $\nabla_h\cdot P_0 = 0$, and it follows
that $(I - P_0)P_0 = 0$, or $P_0^2 = P_0$ and $(I - P_0)^2 = (I - P_0)$; that is,
$P_0$ and $(I-P_0)$ are exact projections on $\fcns$.  

The approximate projection $\Ptw$ uses $\Delta_h$ rather than $\Delta_0$,
\beq   \Ptw v \eq v - \nabla_h  (\Delta_h)^{-1} \nabla_h\cdot v \eeq 
To relate the two, we write
\beq  \Ptw \eq I - \nabla_h\Delta_h^{-1}\nabla_h\cdot
        \eq P_0 + \nabla_h(\Delta_0^{-1} - \Delta_h^{-1})\nabla_h\cdot  \eeq 
and
 \begin{multline}
 \Ptw - P_0 \eq \nabla_h(\Delta_h - \Delta_0)\Delta_h^{-1}\Delta_0^{-1}\nabla_h\cdot
       \eq (\Delta_h - \Delta_0) \Delta_h^{-1} (\nabla_h\Delta_0^{-1}\nabla_h\cdot)
       \eq A(I - P_0) 
\end{multline}
where
$    A \eq (\Delta_h - \Delta_0) \Delta_h^{-1}  $
so that 
\beq  \Ptw \eq  P_0 +  A(I-P_0)\,, \quad P_0\Ptw = P_0\,, \quad (I-P_0)\Ptw = A(I-P_0)  \eeq 
The following lemma, proved in the Appendix, tells us that $A$ is bounded.

\begin{lemma}
The operator $A = (\Delta_h - \Delta_0) \Delta_h^{-1}$ on $X_h$ has
$\|A\| \leq K$, with $K$ independent of $h$.
\end{lemma}

We have estimates for the inverse Laplacians with gain of two discrete derivatives:

\begin{lemma}
As operators on $X_h$
\beq   \|\Delta_h^{-1}\| \leq K_0\,,\quad \|D_h\Delta_h^{-1}\| \leq K_1\,,\quad
        \|D_h^2 \Delta_h^{-1}\| \leq K_2\ |\log{h}|  \eeq 
where $D_h$ is any first difference operator and $D_h^2$ is the product of any two,
with constants independent of $h$.  The same is
true for $\Delta_0^{-1}$ on $X_0$ provided the operators $D_h$ are centered differences.
\end{lemma}

The statement for $\Delta_h$ is proved in \cite{smooth}, Cor. 3.2
and in an equivalent form in \cite{bealay}
The case of $\Delta_0$ follows by applying the first case to subgrids with size $2h$.  The
$|\log {h}|$ factor cannot be improved; this can be seen by example.  We will use the
fact that $\nabla_h \Delta_0^{-1}$ is bounded.
This elliptic estimate applies directly to the projections $\Ptw$ and $P_0$, since we can write
\beq  (I - \Ptw)v = \nabla_h(\Delta_h)^{-1}\nabla_h\cdot v =
      \nabla_h(\nabla_h\cdot \Delta_h^{-1})(v - \langle v \rangle)  \eeq 
where $\langle v \rangle$ is the average of $v$, and similarly for $P_0$.
Then from the lemma we have
\beq  \|\Ptw\| \leq K|\log{h}| \,, \qquad  \|P_0\| \leq K|\log{h}|  \eeq 

We will use estimates for the resolvent $R$ of $\Delta_h$ and the $n$th power $S^n$
of the time-stepping operator $S$ for
Crank-Nicolson.  We define
\beq  R = (I - \tfrac{\tau}{2}\Delta_h)^{-1} \,, \quad
S = (I + \tfrac{\tau}{2}\Delta_h)(I - \tfrac{\tau}{2}\Delta_h)^{-1} \eeq 
The following is proved in \cite{smooth}; see (4.12), (4.13), (7.1), (7.2).

\begin{lemma}
As operators on $\fcns$,
\beq  \|R\| \leq K_1\,, \quad \|D_h R\| \leq K_2\tau^{-1/2}\,, \eeq 
\beq  \|S^n\| \leq K_3\,, \quad \|D_h S^n R\| \leq K_4 (n\tau)^{-1/2} \eeq 
where $D_h$ is any first difference and the constants are independent of $h$ and $\tau$.
\end{lemma}

We will need to know that a grid function on the {\it irregular} points near the interface
is almost the discrete divergence of a function which is smaller by a factor of $h$.
The following is proved as Lemma 8.1 in \cite{smooth}.  Related statements are
Lemma 2.2 in \cite{hack}, Lemma 2.10 in \cite{steve}, and Lemmas 2.2 and 2.6 in \cite{bealay}.

\begin{lemma}
Let ${\mathcal T}_h = \{t_n = n\tau\,, 0 \leq t_n \leq T \}$ and assume
$\tau/h$ is constant.  Let ${\mathcal I}_h$ be a subset of $\Omega_h\times{\mathcal T}_h$ such that
each $(x_j,t_n) \in {\mathcal I}_h$ is within $O(h)$ of $\Gamma(t_n)$.  Let $\varphi$ be
a periodic function on $\Omega_h\times{\mathcal T}_h$ which is zero outside of ${\mathcal I}_h$.
Then there are periodic grid functions $\Phi_\nu$, $0 \leq \nu \leq d$ so that
\beq  \varphi \eq \Phi_0 + \sum_{\nu=1}^d D_\nu^{-}\Phi_\nu \quad \mbox{and} \quad
     \|\Phi_\nu\| \leq Kh \|\varphi\| \eeq 
with some constant $K$ independent of $h$, where
$D_\nu^{-}$ is the backward difference in direction $\nu$,
and the norms are the maximum over $t_n \leq T$ as well as $jh \in \Omega_h$.
\end{lemma}

Lemma 2.4 applies directly to the correction terms and the remaining truncation errors
near the interface, since they occur only at grid points withing $O(h)$ of $\Gamma$.
If, for example, $\varphi$ is function of $O(h)$
on the irregular points, we express the conclusion briefly as
$\varphi = O(h^2) + D_hO(h^2)$ or $\varphi = (I + D_h)O(h^2)$.
We might combine this with Lemma 2.3 and use the fact that the operators commute to
conclude that $\|S^nR\varphi(\cdot,t_k)\|$ is $O(h^2(n\tau)^{-1/2})$.

\ssection{Consistency Estimates}
We estimate the error the exact solution makes in satisfying the scheme (1.7)--(1.13).
We will denote the exact velocity and pressure as $v$ and $q$ to distinguish them
from the computed quantities $u$ and $p$.  We will verify that the error has the
form $O(h^2|\log h|) + D_h O(h^2|\log h|)$; the second term means
a difference operator applied to
a grid function with the specified bound.

\begin{lemma} For the exact velocity $v$ and pressure $q$ we have for $n \geq 1$
\beq  v^{n+1} - v^n \eq - \tau  (v\cdot\nabla v)^{n+1/2} - \tau \nabla q^{n+1/2} +
  \tau \Delta v^\nhalf + \tau C_1  + \tau C_7  + \tau\eps^n \eeq 
where the quantities on the right are computed as in (1.9)--(1.13) from $v$ at times
$t_\nm, t_n, t_\np$ with correction terms.  For $n = 0$,
\beq  v^1 - v^0 \eq - \tau  (v\cdot\nabla v)^0 - \tau \nabla q^0 +
  \tau \Delta v^{1/2} + \tau C_1  + \tau C_7  + \tau\eps^0 \eeq
Here
$ \eps^n = E_0^n + D_h E_1^n $
with $E_k^n = O(h^2|\log h|)$ for $n \geq 1$, $k = 0,1$, and $\eps^0 = O(h|\log{h}|)$.
\end{lemma}

If $v$ and $q$ were smooth across the interface, the scheme would have $O(h^2)$
truncation error.  To verify the statement, we will consider
the corrections at the irregular points near the interface and the remaining errors.
Since corrections are made to first order accuracy,
a typical truncation error $\eta$ on the whole grid
will have the form 
\beq \eta \eq \left\{ \begin{array}{ll}
           O(h) & \;\;\mbox{at an irregular point} \\
           O(h^2) & \;\;\mbox{at a regular point} 
          \end{array} \right.    \eeq
Because of Lemma 2.4, we then have $\eta = O(h^2) + D_h O(h^2)$.
  
To estimate the errors, we first assume the force $f$ is known
exactly, so that the corresponding jumps in $v$, $q$, and their derivatives are exact.
Later we consider the effect of errors in $f$.  We will use the
superscript $\ex$ to denote exact quantities at time $t_\nhalf$, to distinguish them
from quantities computed in the scheme from the exact $v$.
The error $(v^\np - v^n)/\tau - C_1 - (v_t)^\ex$ has the form (3.3) since $C_1$ corrects for $[u_t]$
if the interface crosses the grid point, leaving a remainder at such a point of $O(\tau) = O(h)$.
In dealing with other terms we will first neglect such crossings and return to them afterward.

The error in the advection term,
\beq  \lilth v^n\cdot\nabla_h v^n - \lilhalf v^\nm\cdot\nabla_h v^\nm + C_2 
               - (v\cdot\nabla v)^\ex = O(h^2) \eeq
uniformly, since at an irregular point the correction $C_2$ uses
$[Dv]$, $[D^2v]$, leaving remainder $O(h^3/h)$.  Similarly, the error
$ \lilhalf(\Delta_h v^n + \Delta_h v^\np) + C_3 - (\Delta v)^\ex $
has the form (3.3) since the remaining error at an irregular grid point after correcting
with $C_3$ is $O(h^3/h^2)$. 

Because of (3.4), the error in the divergence
\beq  \eps_4 = \nabla_h\cdot(v\cdot\nabla_h v)^\nhalf + C_4 -  \nabla\cdot(v\cdot\nabla v)^\ex \eeq
has the form (3.3), with the correction $C_4$ similar to $C_3$. 
For the exact pressure $q^\ex$ we have $\Delta q^\ex = - \nabla\cdot(v\cdot\nabla v)^\ex$, with the prescribed
jumps in $p$ and $\pa p/ \pa n$, so that
\beq  \Delta_h q^\ex = - \nabla\cdot(v\cdot\nabla v)^\ex + C_5 + \eps_5 \eeq
where $C_5$ corrects the discrete Laplacian with
$[p], [Dp], [D^2 p], [u\cdot\nabla u]$, and $\eps_5$ has the form (3.3).
Combining (3.5) and (3.6) we have
\beq  \Delta_h q^\ex = -\nabla_h\cdot(v\cdot\nabla_h v)^\nhalf - C_4 + C_5 + \eps_4 + \eps_5 \eeq
The pressure $q^h$ corresponding to the scheme is the
solution with mean value zero of the similar equation
\beq  \Delta_h q^h = -\nabla_h\cdot(v\cdot\nabla_h v)^\nhalf - C_4 + C_5 - m \eeq
where $m$ is the average of $-C_4 + C_5$ over the periodic box.
We check that $m$ is $O(h^2)$: In (3.7)
the averages of $\Delta_h q^\ex$ and the $\nabla_h\cdot\;$ term
are zero, since they are differences.
Because $\eps_4 + \eps_5$ has the form (3.3), and the number of irregular points is $O(h^{1-d})$,
its average is $O(h^2)$,
and therefore the same is true for $-C_4 + C_5$.
Now the right sides of (3.7),(3.8) differ by an error of the form (3.3), and
it follows from Lemma 2.2 that 
$q^h - (q^\ex - q^\ex_0) = O(h^2)$ uniformly, where $q^\ex_0$ is the mean
value of $q^\ex$ on the grid.  Then
$\nabla_h q^h - \nabla_h q^\ex = D_hO(h^2)$.
Also $\nabla_h q^\ex + C_6 - \nabla q^\ex = O(h)$ at irregular points,
and thus is of the form (3.3). With $\nabla q^\nhalf = \nabla_h q^h + C_6$,
we combine the last two estimates to conclude that
$\nabla q^\nhalf - \nabla q^\ex$ is also of the form $(D_h + I)O(h^2)$.

For each of the terms $(v\cdot\nabla v)^{n+1/2}$, $\nabla q^{n+1/2}$,
$\Delta v^\nhalf$, at a grid point crossed by the interface during the time interval,
there is an additional correction with the jump in the quantity.  The remaining
error in time discretization at such a point is $O(\tau) = O(h)$, and this error
again has the form (3.3). 

We now consider the effect of errors in the force $f$.  Suppose at first we assume only
that the error in $f$ is $O(h^2)$.  Then the errors in $[Du]$, $]D^2 u]$, determined from (1.2),
are at most $O(h^2)$ and $O(h)$.  Then, for example, the error in $C_3$ is
$O(h^{-2}\cdot h \cdot h^2) = O(h)$, which is again of the form (3.3)
since it occurs only at irregular points.  Similarly the error in correcting
$\Delta_h u$ where the interface crosses is $O(h)$.  The error in $C_1$ or $C_2$ is
$O(h^2)$, and in $C_4$ it is $O(h)$.  These are no larger than corresponding errors
already estimated.

It seems that for the pressure we need the further assumption that the error in $\nabla_{tan} f$
is $O(h^2)$.  The errors in $[p]$, $[Dp]$, $[D^2p]$ are then $O(h^2)$, $O(h^2)$, and $O(h)$,
respectively.  The errors in $[Dp]$, $[D^2p]$ contribute an error to $C_5$ which is $O(h)$.
The error in $[p]$ contributes an error which could be $O(1)$.  However, the form of
$\Delta_h p$ (1.6) is such that the correction from $[p]$ enters $C_5$ as a difference, and 
consequently this error in $C_5$ has the form $D_h b$ for $b = O(h)$.
(The difference structure for this correction was noted in \cite{leelev}.)
This term $b$ at irregular points has the form $(D_h + I)O(h^2)$,
according to Lemma 2.4, and thus the error in
$C_5$ and $\Delta_h p$ is  $(D_h^2 + D_h) O(h^2)$.  By Lemma 2.2 the resulting error
in $p$, and the contribution to $q^h - q^\ex$ above, is $O(h^2|\log h|)$, slightly worse than before.
This error term leads to the log factor in the statement of the lemma.

For $n=0$ the accuracy of the extrapolation in $v\nabla\cdot v + \nabla q$ is only
$O(h)$.  Since $D_hO(h^2) = O(h)$,  $\eps^0$ is at most $O(h|\log{h}|)$.

\ssection{The error in the solution}
We estimate the growth of the error in the velocity.  We set $w^n = u^n - v^n$.
Rather than estimate the maximum of $w$ directly, it seems better to estimate separately
$y^n \equiv P_0 w^n $ and
$z^n \equiv (I-P_0)w^n$.  Of course $\|w^n\| \leq \|y^n\| + \|z^n\|$, but
we cannot bound $\|y^n\|$ by $\|w^n\|$ because of the $\log{h}$ factor in the bound
(2.11) for the projection.
We will see that $P_0$ is useful in estimating the nonlinear terms.

We begin by subtracting the equations (1.8) and (3.1)  for $u^{n+1}$ and $v^{n+1}$.  The corrections
$C_1$, $C_2$, $C_3$, $C_6$, $C_7$ cancel, and the
the advection terms give us
\beq  g \equiv (u\cdot\nabla_h u)^{n+1/2} - (v\cdot\nabla_h v)^{n+1/2}\eeq
where we now use the
superscript $n+1/2$ as a shorthand for the extrapolation in (1.9).  The pressures $p$ and
$q^h$ are defined by the similar equations (1.12), (3.8), so that
$\Delta_h (p - q^h) = -\nabla_h\cdot g$ and the gradient term in
the equation is
\beq   \nabla_h p^\nhalf - \nabla_h q^\nhalf = - \nabla_h\Delta_h^{-1}\nabla_h\cdot g^\nhalf 
            = -(I-\Ptw)g^\nhalf  
\eeq 
We can then combine terms to get
\beq  (u\cdot\nabla_h u)^{n+1/2} - (v\cdot\nabla_h v)^{n+1/2} + \nabla_h p^\nhalf - \nabla_h q^\nhalf
  \eq  \Ptw g^\nhalf \eeq
and thus the equation for $w$ has the simple form
\beq  w^{n+1} - w^n \eq - \tau  \Ptw g^{n+1/2}
  + (\tau/2)(\Delta_h w^{n+1} + \Delta_h w^n) +\tau \eps^n  \eeq
where $\eps^n$ is the error in the $v$-equation (3.1).
We introduce the operators $R$ and $S$ from (2.12)
and rewrite (4.4) as
\beq  w^{n+1} = Sw^n -\tau  R\Ptw g^{n+1/2} + \tau R\eps^n \eeq
For $n = 0$, $w^0 = 0$ and $g^0 =0$ since $u^0 = v^0$, and
\beq  w^1 =  \tau R\eps^0 \eeq

We obtain separate equations for $y = P_0 w$ and $z = (I-P_0) w$
by applying $P_0$ and $(I-P_0)$ through the $w$-equation and
using the identities (2.8) for $P_0\Ptw$ and $(I-P_0)\Ptw$.
We get
\beq  y^{n+1} \eq S y^n - \tau RP_0 g^{n+1/2}
      + \tau P_0 R\eps^n                     \eeq 
and
\beq  z^{n+1} \eq  Sz^n -\tau R A (I-P_0) g^{n+1/2} 
      + \tau (I-P_0)R\eps^n \eeq
and for $n = 0$ 
\beq  y^1 =  \tau P_0 R\eps^0 \,, \qquad z^1 =  \tau (I-P_0)R\eps^0 \eeq 

To estimate the growth of the error, we define
\beq   \delta^n = \max_{1 \leq m \leq n} \left(\|y^m\| + \|z^m\| \right) \,,\quad  n \geq 1
   \eeq
We will prove by induction that 
\beq   \delta^n \leq K_T h^2(\log{h})^2 \,, \qquad  n\tau \leq T  \eeq
for some constant $K_T$ and for $h$ sufficiently small.  We will then have
verified the estimate (1.14) for the error $w$ in velocity.
Once the estimate is proved for some $n$,
it follows that $\|D_h w^m\| \leq 1$ for $m \leq n$ and for sufficiently small $h$;
we will use this below for the nonlinear term in the error.

To estimate $g^\nhalf$ we write $g = g_1 + g_2 + g_3$ with
\beq   g_1 = v\cdot\nabla_h w\,, \quad 
    g_2 = w\cdot\nabla_h v\,, \quad
    g_3 = w\cdot\nabla_h w   \eeq
It will be important that $D_h v$ is uniformly bounded for any first difference $D_h$
since $v$ is continuous at the interface.  We will use the notation $B(w)$ for any
bounded linear operator applied to $w$; that is, $\|B(w)\| \leq K\|w\|$
with constant $K$ independent of $h$.  Thus, for example, we can write the difference
of a product $v_jw_i$ as
\beq   D_h(v_jw_i) = v_j D_h w_i + B(w)  \eeq
since $D_h v$ is bounded.  Then for $g_1$ and $g_2$ we have
\beq   g_1^\nhalf = D_hB(w^n) + B(w^n) +  D_hB(w^\nm) + B(w^\nm)   \,, 
           \quad g_2^\nhalf = B(w^n) + B(w^\nm) \eeq
For the nonlinear term, we can assume by induction, as remarked above, that
$\|D_hw^n\| = O(1)$, and the same for $\|w^\nm\|$, so that
$\|g_3\| \leq K(\|w^n\| + \|w^\nm\|) $. 

The main difficulty in estimating $y^n$ is the effect of  $P_0$ on $g$,
since $P_0$ has norm $O(|\log{h}|)$.  Applying it directly would lose stability.
Since we have already estimated $g$,
it is equivalent to estimate $P_0g$ or $(I-P_0)g$, and we choose the latter,
$(I-P_0)g = \nabla_h \Delta_0^{-1} \nabla_h\cdot g$.
We note, using Lemma 2.2, that
$\nabla_h\Delta_0^{-1}$ is a bounded operator, since $\nabla_h$ is a centered difference,
and for some terms in $g$ it will be enough to estimate the divergence.

With $g$ as in (4.12), we begin with $g_1$.
Writing $(I-P_0)g_1 = (I-P_0)( v\cdot\nabla_h y +   v\cdot\nabla_h z)$,
we treat the first term by
using $\nabla_h\cdot y = 0$.  We apply $\nabla_h\cdot$ to $v\cdot\nabla_h y$ obtaining
(with sum over $j$, and $D_j$ the centered difference for $x_j$, $1\leq j\leq d$),
\beq  \nabla_h\cdot (v_jD_j y) = D_j\nabla_h\cdot(v_j y) + \nabla_h\cdot B(y)
  = D_j(v_j\nabla_h\cdot y) + (D_j + \nabla_h\cdot)B(y) = 0 + D_hB(y) \eeq
so that by the remark above
\beq  (I-P_0) (v\cdot\nabla_h y) = D_hB(y)  \eeq

For the second term in $g_1$, $z$ is in the range of $(I-P_0)$ and thus is a discrete
gradient, 
so that $D_jz_i = D_iz_j$.  Then the divergence is (with sum over $i$, $j$)
\beq  D_i(v_j D_jz_i) = D_i(v_j D_iz_j)
      = D_i^2(v_jz_j) + D_iB(z) = \Delta_0(v_jz_j) + D_iB(z) \eeq
Then $(I-P_0)(v\cdot\nabla_h z) = \nabla_h\Delta_0^{-1}\Delta_0(v_jz_j) + D_h B(z)
= \nabla_h (v_jz_j) + D_h B(z) = D_hB(z)$.

The term $(I-P_0)g_2$ has the form $\nabla_h\Delta_0^{-1} \nabla_h\cdot (w\cdot\nabla_h v)
= D_hB(w)$.  For $g_3$, again by induction $\nabla_h w = O(1)$, and we can treat
$(I-P_0)g_3$ like $(I-P_0)g_2$.  In summary we have shown that
\beq  P_0 g^\nhalf = \Phi_0^n + D_h\Phi_1^n + \Phi_0^\nm + D_h\Phi_1^\nm \eeq
where
\beq   \|\Phi_k^m\| \leq K (\|y^m\| + \|z^m\|)\,, \qquad k = 0,1;\;\;m= n-1,n \eeq
and $(I-P_0)g$ has the same form.

We are now ready to prove (4.11) by induction.  
Since $\eps^0 = O(h|\log{h}|)$, it is evident from (4.9) that (4.11) holds for $n=1$.
We assume it is true for $n$ and prove it for $n+1$.  Here and below we use the fact
that $\|P_0\| \leq K|\log{h}|$.
From the $y$-equation (4.7) we have
\beq  y^\np \eq -\tau \sum_{\ell=1}^n S^{n-\ell} RP_0 g^{\ell+1/2} \p 
        \tau \sum_{\ell=0}^n S^{n-\ell} RP_0 \eps^\ell
          \,\equiv\, \Sigma_1 \p \Sigma_2  \eeq
For $\Sigma_1$ we use (4.18),(4.19) and (2.14) to estimate for $\ell \leq n-1$
\beq  |S^{n-\ell} D_h R\Phi_1^\ell| \leq K ((n-\ell)\tau))^{-1/2} \delta^\ell \eeq
and similarly for other terms, while for $\ell = n$ we use (2.13) to get
\beq  |D_h R\Phi_1^n| \leq K \tau^{-1/2} \delta^n \eeq
Then
\beq  |\Sigma_1| \leq K_1 \bigg( \sum_{\ell=1}^\nm ((n-\ell)\tau))^{-1/2} \delta^\ell \tau
         \p \tau^{-1/2} \delta^n \tau \bigg) \eeq
For $\Sigma_2$ we recall from Sec. 3 that $\eps^n = (I + D_h)O(h^2|\log{h}|)$ for $n\geq1$,
and we again use (2.14) for $1 \leq \ell \leq n-1$
and (2.13) for $\ell = n$.  We treat $\ell = 0$ separately using $\eps^0 = O(h|\log{h}|)$
and (2.14).  Then
$|\Sigma_2|$ is bounded by a constant times
\beq  \sum_{\ell=1}^\nm ((n-\ell)\tau))^{-1/2} h^2(\log{h})^2 \tau
         \p \tau^{-1/2} h^2(\log{h})^2  \tau
          \p  h(\log{h})^2  \tau \lee K h^2(\log{h})^2  \eeq
since the sum approximates an integrable function of time.
Estimates for $z^{n+1}$ are entirely similar, since $A$ is bounded independent
of $h$, according to Lemma 2.1, and adding the two
inequalities gives
\beq  \delta^\np \lee K_1  \sum_{\ell=1}^\nm ((n-\ell)\tau))^{-1/2} \delta^\ell \tau
         \p K_1 \tau^{-1/2} \delta^n \tau \p K_2 h^2(\log{h})^2 \eeq
To simplify this we use H\"older's inequality to obtain 
\beq  \left(\delta^\np\right)^3 \lee 
   K_1' \bigg(\sum_{\ell=1}^\nm ((n-\ell)\tau))^{-2/3} \tau + \tau^{-2/3}\tau \bigg)^2
       \sum_{\ell=1}^n (\delta^\ell)^3 \tau \p K_2' (h^2(\log{h})^2)^3  \eeq
The first sum is uniformly bounded, and the inequality
has the form 
\beq  \kappa^\np \lee A \sum_{\ell=1}^n \kappa^\ell \tau + B \eeq
with $\kappa^n = (\delta^n)^3$ and $B = (h^2(\log{h})^2)^3 $.
It follows easily from this and the estimate for $\delta^1$ that
$\delta^m$ has the bound (4.11) for $m \leq \np$, provided $(n+1)\tau \leq T$.

Finally we discuss the accuracy of the pressure.
We can compute the pressure $p^n$ at time $t_n$ as in (1.12)
using $(u\cdot\nabla_h u)^n$ and the corrections $C_4$, $C_5$.
As in Sec. 3, $q^h - (q - q_0) = O(h^2|\log{h}|)$,
where $q$ is the exact pressure, $q^h$ is computed from the exact velocity $v$
and corrections as in (3.8) at time $t_n$,
and $q_0$ is the mean value of $q$ on the grid.
Combining the estimate (4.11) or (1.14)
for $w$ with (4.14) for $g$, we now have
$ g = (D_h + I)O(h^2|\log {h}|^2 )$.  Lemma 2.2 then gives
\beq  p^n - q^h \eq \Delta_h^{-1} \nabla_h\cdot  (D_h + I)O(h^2|\log {h}|^2)
      = O(h^2|\log {h}|^3 )\eeq
Thus  $p^n - (q - q_0) =  O(h^2|\log {h}|^3)$, as stated.

\appendix
\ssection{Appendix}
We prove a simple criterion for the boundedness of Fourier multiplier operators on
periodic grid functions in maximum norm.  We use this statement to prove Lemma 2.1.
Related statements are given in \cite{btw}, Ch.1 and \cite{sulibook}, Sec. 2.5.1. 

We will assume  $L = \pi$ for convenience and $\pi/h$ is an
integer. Let
$I_d$ be the set of integer $d$-tuples $j$ with $|j_\nu| \leq \pi/h$, $1 \leq \nu \leq d$.
Then $\Omega_h = hI_d$ and $\fcns$ is the space of periodic grid functions.  
For $\varphi \in \fcns$ we have the discrete Fourier
transform and inverse
\beq \phihat(k) \eq \sum_{j \in I^d} \varphi(jh)e^{-ikjh} \,, \qquad
      \varphi(jh) \eq (2\pi)^{-d} \sum_{k \in I^d} \phihat(k) e^{ikjh} h^d  \eeq
and the isometry
\beq \sum_{j \in I^d} |\varphi(jh)|^2 \eq 
         (2\pi)^{-d} \sum_{k \in I^d}|\phihat(k)|^2 h^d \,. \eeq

\begin{lemma} 
Suppose an operator $A$ is defined on  $\fcns$ by
\beq (A\varphi){\hat \;} (k) \eq \sig(kh)\phihat(k)  \eeq
where $\sig$ is a function of $\xi \in \R^d$, with period $2\pi$ in
each $\xi_\nu$, $1 \leq \nu \leq d$.
Let $\|A\|$ be the norm of $A$ as an operator on $\fcns$ with maximum norm.
Then
\beq \|A\| \lee K \bigg( \sum_{k \in I^d}
     \left(|(D_\nu^+)^s \sig(kh)|^2 \p |\sig(kh)|^2\right) h^d \bigg)^{d/4s}
    \bigg(\sum_{k \in I^d}|\sig(kh)|^2 h^d \bigg)^{(1 - d/2s)/2}
\eeq
where $s$ is an integer, $s > d/2$, $(D_\nu^+)^s$ is the forward divided difference
operator in direction $\nu$, a sum over $1 \leq \nu \leq d$ is implied,
and $K$ is independent of $h$.
\end{lemma}

\begin{proof}
We can write $A$ as a discrete convolution with the inverse transform
$a_j$ of $\sigma(kh)$,
\beq
A\varphi \eq \sum_{\ell \in I^d} a_{j-\ell} \varphi(\ell h) \,, \qquad
a_j \eq (2\pi)^{-d} \sum_{k \in I^d} \sig(kh) e^{ikjh} h^d  \eeq
so that the operator norm of $A$ on $\fcns$  has the bound
\beq \|A\| \lee \sum_{j \in I^d} |a_j| \,. \eeq
We will temporarily extend this sum to all 
$j \in \Z^d$ with $a_j = 0$ for $j \notin I^d$ and estimate as in
the proof of Sobolev's theorem and particularly as in
\cite{btw}, Ch. 1, Thm. 3.1.    With $R$ to be chosen we have
\begin{multline}
    \sum_{j \in \Z^d} |a_j| \eq \sum_{|j| \geq R} |a_j||j|^s|j|^{-s} 
       \p \sum_{|j| < R} |a_j|  \\
          \lee  \bigg(\sum_{|j| \geq R} |j|^{-2s} \bigg)^{1/2}
          \bigg(\sum_{|j| \geq R} |a_j|^2|j|^{2s} \bigg)^{1/2}
       \p  \bigg(\sum_{|j| \leq R} 1 \bigg)^{1/2} 
         \bigg(\sum_{|j| \leq R} |a_j|^2 \bigg)^{1/2}  \\          
   \lee  \left(\int_{|x| \geq R - \sqrt{d}}|x|^{-2s}\,dx\right)^{1/2} M_1
         \p  \left(\int_{|x| \leq R + \sqrt{d}}\,dx\right)^{1/2} M_0 \\
   \lee K_1(R - \sqrt{d})^{(-2s + d)/2} M_1 \p 
             K_0(R + \sqrt{d})^{d/2} M_0   
\end{multline}
where
\beq M_0^2 \eq \sum_{j \in \Z^d} |a_j|^2 \,, \qquad
     M_1^2 \eq \sum_{j \in \Z^d} |a_j|^2(|j|^{2s} + 1) \,. \eeq
We choose $R = \sqrt{d} + (M_1/M_0)^{1/s}$,
so that $R \geq \sqrt{d} + 1$ and  $(R + \sqrt{d})/(R - \sqrt{d})$ is
bounded above.  Then both terms at the end of (A.7) have  $M_1^{d/2s}M_0^{1 - d/2s}$,
and the inequality simplifies to
\beq \sum_{j \in \Z^d} |a_j| \lee  K
            \bigg(\sum_{j \in \Z^d} |a_j|^2(1 + |j|^{2s})\bigg)^{d/4s}
             \bigg(\sum_{j \in \Z^d} |a_j|^2\bigg)^{(1 - d/2s)/2} \,. \eeq

Next we relate $|j|^{s}a_j$ to $(D_\nu^+)^s\sigma(kh)$.
A summation by parts using periodicity shows that
\beq (2\pi)^{-d}\sum_{k \in I^d} (D_\nu^+)^s \sig(kh) e^{ikjh} h^d
     \eq \beta(j_\nu h,h)^s a_j \,, \qquad \beta = (e^{-ij_\nu h} - 1)/h \eeq
so that $D_\nu^+ \sig(kh)$ is the transform of the last term, and thus by (A.2)
\beq \sum_{j \in I^d} |\beta(j_\nu h,h)|^{2s} |a_j|^2 \eq
          (2\pi)^{-d}\sum_{k \in I^d} |(D_\nu^+)^s \sig(kh)|^2 h^d \,.  \eeq
We note that
$ |e^{-ij_\nu h} - 1| = 2|\sin(j_\nu h/2)|\gee c|j_\nu h| $
since $|j_\nu h)| \leq \pi$.  Thus $\sum_j |j_\nu|^{2s} |a_j|^2$ is
bounded by the right side of (A.11), and 
summing over $\nu$ we get
\beq \sum_{j \in I^d} |j|^{2s} |a_j|^2 \lee
           K \sum_{\nu=1}^d \sum_{k \in I^d} |(D_\nu^+)^s \sig(kh)|^2 h^d \,.  \eeq
Finally, combining (A.6),(A.9),(A.12) gives the conclusion (A.4).
\end{proof}

\bigskip

{\it Proof of Lemma 2.1.}  For simplicity we assume $L = \pi$.
From (2.2), (2.3) the Fourier symbol of $A$ in case $d=2$ is
\beq  \sigma(\xi) \eq (s_1^4 + s_2^4)/(s_1^2 + s_2^2)  \,, \qquad s_j = \sin(\xi_j/2)  \eeq
where $\xi = kh$, and similarly for $d=3$.
We set $\sigma(0) = 0$ so that $A$ is extended to all of $\fcns$.
For $|\xi_j)| \leq \pi$, $|s_j| \geq c|\xi_j|$, and $\sigma$ and $\pa\sigma/\pa \xi_j$
are bounded.  It is easy to check
that $\pa^2\sigma/\pa \xi_j^2$ is bounded for $\xi \neq 0$. It follows that
the second difference
$D_j^2 \sigma$ is bounded at grid points not adjacent to $0$.  However, for
grid points near $0$, $\sigma = O(h^2)$, so that $D_j^2 \sigma$ is bounded in that
case also.  Thus the right side of (5.4) is bounded independent of $h$ with $s = 2$, 
and Lemma A.1 ensures that the same is true for $\|A\|$.

\bibliographystyle{amsplain}

\end{document}